\documentclass[12pt]{amsart}
\usepackage{amsfonts,amssymb,amscd,amsmath,enumerate,verbatim,calc,latexsym,pstcol,pst-plot,pst-3d,}
\def\NZQ{\mathbb}               

\def\ZZ{{\NZQ Z}}
\def\RR{{\NZQ R}}
\def\CC{{\NZQ C}}

\def\frk{\mathfrak}               

\def\Phi{{\frk N}}
\def\ab{{\bold a}}

\def\eb{{\bold e}}

\def\opn#1#2{\def#1{\operatorname{#2}}} 
\opn\chara{char} 
\opn\length{\ell} 
\opn\pd{pd} 
\opn\rk{rk}
\opn\projdim{proj\,dim} 
\opn\injdim{inj\,dim} 
\opn\rank{rank}
\opn\depth{depth} 
\opn\grade{grade} 
\opn\height{height}
\opn\embdim{emb\,dim} 
\opn\codim{codim}

\opn\Tr{Tr} 
\opn\bigrank{big\,rank}
\opn\superheight{superheight}
\opn\lcm{lcm}
\opn\trdeg{tr\,deg}
\opn\reg{reg} 
\opn\lreg{lreg} 
\opn\ini{in} 
\opn\lpd{lpd}
\opn\size{size}
\opn\mult{mult}
\opn\dist{dist}
\opn\cone{cone}
\opn\lex{lex}
\opn\rev{rev}
\opn\div{div} \opn\Div{Div} \opn\cl{cl} \opn\Cl{Cl}
\opn\Spec{Spec} \opn\Supp{Supp} \opn\supp{supp} \opn\Sing{Sing}
\opn\Ass{Ass} \opn\Min{Min}
\opn\Ann{Ann} \opn\Rad{Rad} \opn\Soc{Soc}
\opn\Syz{Syz} \opn\Im{Im} \opn\Ker{Ker} \opn\Coker{Coker}
\opn\Am{Am} \opn\Hom{Hom} \opn\Tor{Tor} \opn\Ext{Ext}
\opn\End{End} \opn\Aut{Aut} \opn\id{id} \opn\ini{in}

\opn\nat{nat}
\opn\pff{pf}
\opn\Pf{Pf} \opn\GL{GL} \opn\SL{SL} \opn\mod{mod} \opn\ord{ord}
\opn\Gin{Gin}
\opn\Hilb{Hilb}\opn\adeg{adeg}\opn\std{std}\opn\ip{infpt}
\opn\Pol{Pol}
\opn\sat{sat}
\opn\Var{Var}
\opn\Gen{Gen}

\opn\aff{aff} \opn\con{conv} \opn\relint{relint} \opn\st{st}
\opn\lk{lk} \opn\cn{cn} \opn\core{core} \opn\vol{vol}
\opn\link{link} \opn\star{star}
\opn\gr{gr}

\def\Hc{{\mathcal H}}

\def\Pc{{\mathcal P}}
\def\Qc{{\mathcal Q}}

\def\pot#1#2{#1[\kern-0.28ex[#2]\kern-0.28ex]}

\opn\dirlim{\underrightarrow{\lim}}
\opn\inivlim{\underleftarrow{\lim}}
\let\to=\rightarrow

\def\Implies{\ifmmode\Longrightarrow \else
        \unskip${}\Longrightarrow{}$\ignorespaces\fi}
\def\implies{\ifmmode\Rightarrow \else
        \unskip${}\Rightarrow{}$\ignorespaces\fi}
\def\iff{\ifmmode\Longleftrightarrow \else
        \unskip${}\Longleftrightarrow{}$\ignorespaces\fi}

\let\:=\colon
\newtheorem{Theorem}{Theorem}[section]
\newtheorem{Lemma}[Theorem]{Lemma}

\newtheorem{Example}[Theorem]{Example}

\let\epsilon\varepsilon
\let\phi=\varphi
\let\kappa=\varkappa
\def\qed{\ifhmode\textqed\fi
      \ifmmode\ifinner\quad\qedsymbol\else\dispqed\fi\fi}
\def\textqed{\unskip\nobreak\penalty50
       \hskip2em\hbox{}\nobreak\hfil\qedsymbol
       \parfillskip=0pt \finalhyphendemerits=0}
\def\dispqed{\rlap{\qquad\qedsymbol}}

\opn\dis{dis}
\opn\height{height}
\opn\dist{dist}
\def\pnt{{\raise0.5mm\hbox{\large\bf.}}}

\opn\Lex{Lex}

\begin{document}
\title{
Roots of Ehrhart polynomials of Gorenstein Fano polytopes
}
\author{Takayuki Hibi,
Akihiro Higashitani
and
Hidefumi Ohsugi
}
\address{Takayuki Hibi,
Department of Pure and Applied Mathematics,
Graduate School of Information Science and Technology,
Osaka University,
Toyonaka, Osaka 560-0043, Japan}
\email{hibi@math.sci.osaka-u.ac.jp}
\address{Akihiro Higashitani,
Department of Pure and Applied Mathematics,
Graduate School of Information Science and Technology,
Osaka University,
Toyonaka, Osaka 560-0043, Japan}
\email{sm5037ha@ecs.cmc.osaka-u.ac.jp}
\address{Hidefumi Ohsugi,
Department of Mathematics,
College of Science,
Rikkyo University,
Toshima-ku, Tokyo 171-8501, Japan} 
\email{ohsugi@rkmath.rikkyo.ac.jp}
\begin{abstract}
Given arbitrary integers $k$ and $d$ 
with $0 \leq 2k \leq d$, we construct 
a Gorenstein Fano polytope
$\Pc \subset \RR^d$ of dimension $d$ such that
(i) its Ehrhart polynomial
$i(\Pc, n)$ possesses $d$ distinct roots;
(ii) $i(\Pc, n)$ possesses exactly $2k$ imaginary roots;
(iii) $i(\Pc, n)$ possesses exactly $d - 2k$ real roots;
(iv) the real part of each of the imaginary roots
is equal to $- 1 / 2$; 
(v) all of the real roots belong to the open interval
$(-1, 0)$. 
\end{abstract}
\maketitle
%
%
%
Recently, many research papers on convex polytopes,
including
\cite{BDDPS}, \cite{BHW}, \cite{Bra}, \cite{BD}, 
\cite{HSW} and \cite{Pfe},
discuss roots of Ehrhart polynomials.
One of the fascinating topics is
the study on roots of Ehrhart polynomials
of Gorenstein Fano polytopes.

Let $\Pc \subset \RR^N$ be an integral convex polytope
of dimension $d$ and $\partial \Pc$ its boundary.
(An integral convex polytope is a convex polytope
all of whose vertices have integer coordinates.)
Given integers $n = 1, 2, \ldots$, we write 
$i(\Pc, n)$ for the number of integer points belonging to
$n \Pc$, where $n \Pc = \{ n \alpha : \alpha \in \Pc \}$.
In other words,
\[
i(\Pc, n) = | n \Pc \cap \ZZ^N |,
\, \, \, \, \,
\, \, \, \, \,
n = 1, 2, \ldots.
\]
Late 1950's Ehrhart did succeed in proving that $i(\Pc, n)$ is
a polynomial in $n$ of degree $d$ with $i(\Pc, 0) = 1$.
We call $i(\Pc, n)$ the {\em Ehrhart polynomial} of $\Pc$. 
Ehrhart's ``\,loi de r\'eciprocit\'e\,'' guarantees that 
\[
( - 1 )^d i(\Pc, - n) 
= | n (\Pc \setminus \partial \Pc) \cap \ZZ^N |, 
\, \, \, \, \,
\, \, \, \, \,
n = 1, 2, \ldots.
\]

We define the sequence
$\delta_0, \delta_1, \delta_2, \ldots$ of integers
by the formula
\begin{eqnarray*}
\label{delta}
(1 - \lambda)^{d + 1}
\left[ 1 + \sum_{n=1}^{\infty} i(\Pc,n) \lambda^n \right]
= \sum_{i=0}^{\infty} \delta_i \lambda^i.
\end{eqnarray*}
Since $i(\Pc, n)$ is
a polynomial in $n$ of degree $d$ with $i(\Pc, 0) = 1$,
a fundamental fact on generating functions guarantees that
$\delta_i = 0$ for every $i > d$.
The sequence
\[
\delta(\Pc) = (\delta_0, \delta_1, \ldots, \delta_d)
\]
is called the
{\em $\delta$-vector} of $\Pc$.
Thus $\delta_0 = 1$,
$\delta_1 = |\Pc \cap \ZZ^N| - (d + 1)$
and
$\delta_d = |(\Pc - \partial \Pc) \cap \ZZ^N|$.
Each $\delta_i$ is nonnegative
(Stanley \cite{StanleyDRCP}).
If $\delta_d \neq 0$, then
$\delta_1 \leq \delta_i$ for every $1 \leq i < d$
(\cite{HibiLBT}).
We refer the reader to \cite{Ehrhart},
\cite{HibiRedBook}, \cite{StanleyEC}
\cite{StanleyJPAA},
\cite{StanleyEJC} and  
\cite{StanleyGreenBook}
for further informations on Ehrhart polynomials
and $\delta$-vectors.

A {\em Fano polytope} is an integral convex polytope
$\Pc \subset \RR^d$ of dimension $d$ 
such that
the origin of $\RR^d$ is a unique integer point
belonging to the interior $\Pc \setminus \partial \Pc$
of $\Pc$. 
A Fano polytope is called {\em Gorenstein} 
if its dual polytope is integral.
(Recall that the dual polytope $\Pc^\vee$
of a Fano polytope $\Pc$ is the convex polytope
which consists of those $x \in \RR^d$
such that $\langle x, y \rangle \leq 1$ for all
$y \in \Pc$, where $\langle x, y \rangle$
is the usual inner product of $\RR^d$.)

Let $\Pc \subset \RR^d$ be a Fano polytope with 
$\delta(\Pc) = (\delta_0, \delta_1, \ldots, \delta_d)$
its $\delta$-vector.  
It follows from \cite{Batyrev} and \cite{Hibidual}
that the following conditions 
are equivalent:
\begin{itemize}
\item
$\Pc$ is Gorenstein;
\item
$\delta(\Pc)$ is symmetric, i.e., 
$\delta_i = \delta_{d-i}$ for every $0 \leq i \leq d$;
\item
$i(\Pc, n) = ( - 1 )^{d} i(\Pc, - n - 1)$.
\end{itemize}

Let $\Pc \subset \RR^N$ be an integral convex polytope
of dimension $d$ and $i(\Pc, n)$ its Ehrhart polynomial.
A complex number $a \in \CC$ is called a {\em root} of
$i(\Pc, n)$ if $i(\Pc, a) = 0$.
Let $\Re(a)$ denote the real part of $a \in \CC$.
An outstanding conjecture given in \cite{BDDPS} says that
every root $a \in \CC$ of $i(\Pc, n)$ satisfies 
$- d \leq \Re(a) \leq d - 1$.

When $\Pc \subset \RR^d$ is a Gorenstein Fano polytope,
since $i(\Pc, n) = ( - 1 )^{d} i(\Pc, - n - 1)$,
the roots of $i(\Pc, n)$ locate symmetrically 
in the complex plane 
with respect to the line $\Re(z) = - 1 / 2$.
Thus in particular, if $d$ is odd, then $- 1 / 2$
is a root of $i(\Pc, n)$.
It is known \cite[Proposition 1.8]{BHW}
that,
if all roots $a \in \CC$ of 
$i(\Pc, n)$ of an integral convex polytope 
$\Pc \subset \RR^d$ of dimension $d$ 
satisfy $\Re(a) = - 1 / 2$, then 
$\Pc$ is unimodular isomorphic to 
a Gorenstein Fano polytope
whose (usual) volume is at most $2^d$.


\begin{Theorem}
\label{main}
Given arbitrary nonnegative integers $k$ and $d$ 
with $0 \leq 2k \leq d$, there exists
a Gorenstein Fano polytope
$\Pc \subset \RR^d$ of dimension $d$ such that
\begin{enumerate}
\item[{\rm (i)}]
$i(\Pc, n)$ possesses $d$ distinct roots;
\item[{\rm (ii)}]
$i(\Pc, n)$ possesses exactly $2k$ imaginary roots;
\item[{\rm (iii)}]
$i(\Pc, n)$ possesses exactly $d - 2k$ real roots;
\item[{\rm (iv)}]
the real part of each of the imaginary roots
is equal to $- 1 / 2$; 
\item[{\rm (v)}]
all of the real roots belong to the open interval
$(-1, 0)$.
\end{enumerate}
\end{Theorem}

\begin{proof}
Let $\eb_1, \ldots, \eb_d$ 
denote the canonical unit vectors
of $\RR^d$.  
Let $\Qc \subset \RR^d$ be the convex polytope
which is the convex hull of
$\eb_1, \ldots, \eb_{2k}$ and
$-(\eb_1 + \cdots + \eb_{2k})$.  Then
$\Qc$ is an integral convex polytope of
dimension $2k$ with
$\delta(\Qc) = (1, 1, \ldots, 1) \in \ZZ^{2k+1}$.
Let $\Qc^{c} \subset \RR^d$ be the convex polytope
which is the convex hull of
$\Qc \cup \{ \eb_{2k+1}, \ldots, \eb_{d} \}$. 
Then $\delta(\Qc^c) = (\delta(\Qc), 0, \ldots, 0) \in \ZZ^{d+1}$.
Hence the convex polytope
$(d - 2k + 1) \Qc^c$ possesses a unique integer point 
$\ab$ in its interior.
Now, write $\Pc \subset \RR^d$ for the integral convex polytope
$(d - 2k + 1) \Qc^c - \ab$.
Then $\Pc$ is a Gorenstein Fano polytope.
Our work is to show that $\Pc$ enjoys the required properties
(i) -- (v).

Since 
\[
\sum_{n=0}^{\infty} i(\Qc^c, n) \lambda^n
= \frac{1 + \lambda + \lambda^2 + \cdots + \lambda^{2k}}
{(1 - \lambda)^{d+1}},
\]
one has
\begin{eqnarray*}
i(\Qc^c, n) & = & \sum_{i=n-2k}^{n} {d + i \choose d} 
  =  \sum_{i=0}^{2k} {d + (n - 2k) + i \choose d} \\
& = & \sum_{i=0}^{2k} {d + n - (2k - i) \choose d} 
  =  \sum_{i=0}^{2k} {n + d - i \choose d} \\
& = & \sum_{i=0}^{2k} \left({n + d - i + 1 \choose d + 1} 
- {n + d - i \choose d + 1} \right) \\
& = &
{n + d + 1 \choose d + 1} 
- {n + d - 2k \choose d + 1} \\
& = &
\frac{1}{(d+1)!}
(\prod_{i=1}^{d - 2k}(n + i))
(\prod_{i=0}^{2k}(n + d + 1 - i) 
- \prod_{i=0}^{2k}(n - i)).
\end{eqnarray*}
Since 
\[
i(\Pc, n) = i((d - 2k + 1) \Qc^c, n)
= i(\Qc^c, (d - 2k + 1)n), 
\]
one has 
\begin{eqnarray*}
i(\Pc, n) 
& = &
\frac{(d - 2k + 1)^{d+1}}{(d+1)!}
(\prod_{i=1}^{d - 2k}(n + \frac{i}{d - 2k + 1})) F(n),
\end{eqnarray*}
where
\begin{eqnarray*}
F(n) & = & \prod_{i=0}^{2k}(n + \frac{d + 1 - i}{d - 2k + 1})
- \prod_{i=0}^{2k}(n - \frac{i}{d - 2k + 1}) \\
& = &
\prod_{i=0}^{2k}(n + \frac{d + 1 - (2k - i)}{d - 2k + 1})
- \prod_{i=0}^{2k}(n - \frac{i}{d - 2k + 1}).
\end{eqnarray*}
Now, since
\[
- \frac{d + 1 - (2k - i)}{d - 2k + 1}
< - 1 / 2 < 
\frac{i}{d - 2k + 1}
\]
and since
\[
- \frac{d + 1 - (2k - i)}{d - 2k + 1}
+ \frac{i}{d - 2k + 1} = - 1,
\]
Lemma \ref{root} below guarantees that
$F(n)$ possesses $2k$ distinct roots
and each of them
is an imaginary root with $- 1 / 2$ its real part.
Finally, the real roots of
$i(\Pc, n)$ are 
\[
- \frac{i}{d - 2k + 1}, 
\, \, \, \, \, 
\, \, \, \, \, 
1 \leq i \leq d - 2k,
\]
which belong to the open interval $(- 1, 0)$. 
\, \, \, \, \, \, \, \, \, \, 
\, \, \, \, \, \, \, \, \, \, 
\, \, \, \, \, 
\end{proof}

\begin{Lemma}
\label{root}
Let $\alpha_0, \alpha_1, \ldots, \alpha_{2k}$ 
and $\beta_0, \beta_1, \ldots, \beta_{2k}$
be rational numbers satisfying 
$\alpha_i < - 1 /2 < \beta_i$ 
and $\alpha_i + \beta_i = - 1$
for all $i$.  Let
\[
f(x) = \prod_{i=0}^{2k} (x - \alpha_i)
- \prod_{i=0}^{2k} (x - \beta_i)
\]
be a polynomial in $x$ of degree $2k$.
Then $f(x)$ possesses $2k$ distinct roots and
each of them
is an imaginary root with $- 1 / 2$ its real part.
\end{Lemma}

\begin{proof}
We employ a basis technique appearing in \cite{RV}.
Let $a \in \CC$
with $\Re(a) > - 1 / 2$.
Since $\alpha_i < - 1 /2 < \beta_i$
and $(\alpha_i + \beta_i) / 2 = - 1 / 2$, 
it follows that
$|a - \alpha_i| > |a - \beta_i|$.
Thus $\prod_{i=0}^{2k} |a - \alpha_i|
> \prod_{i=0}^{2k} |a - \beta_i|$.
Hence $f(a) \neq 0$.
Similarly, 
if $a \in \CC$ with $\Re(a) < - 1 / 2$, then
$|a - \alpha_i| < |a - \beta_i|$
for all $i$.  Thus 
$\prod_{i=0}^{2k} |a - \alpha_i|
< \prod_{i=0}^{2k} |a - \beta_i|$.
Hence $f(a) \neq 0$.
Consequently, all roots $a \in \CC$ of $f(x)$
satisfy $\Re(a) = - 1 / 2$.

%
%
%

Substituting $y = x + 1 / 2$ and 
$\gamma_i = \beta_i + 1 / 2$
in $f(x)$, it follows 
that each of the roots $a \in \CC$
of the polynomial 
\[
g(y) = \prod_{i=0}^{2k} (\gamma_i + y)
+ \prod_{i=0}^{2k} (\gamma_i - y)
\]
in $y$ of degree $2k$
satisfied $\Re(a) = 0$.  
Since $\gamma_i > 0$, one has $g(0) \neq 0$.
Hence $g(y)$ possesses no real root.  Thus
all roots of $f(x)$ are imaginary roots.

What we must prove is that
$g(y)$ possesses $2k$ distinct roots.
Let $b \in \RR$ and
$\theta_i(b)$ the argument of
$\gamma_i + b\sqrt{-1}$,
where $- \pi / 2 < \theta_i(b) < \pi / 2$.
Then $b\sqrt{-1}$ is a root of $g(y)$ 
if and only if
\[
\prod_{i=0}^{2k} e^{\sqrt{-1}\,\theta_i(b)}
= - \prod_{i=0}^{2k} e^{- \sqrt{-1}\,\theta_i(b)}.
\]
In other words, $b\sqrt{-1}$ is a root of $g(y)$ 
if and only if
\[
\prod_{i=0}^{2k} e^{2 \sqrt{-1}\,\theta_i(b)}
= - 1,
\]
which is equivalent to saying that
\[
\sum_{i=0}^{2k} \theta_i(b) \equiv \frac{\pi}{2}
\, \, \, \, \, 
(\mod \, \, \pi).
\]
Now, we study the function
$h(y) = \sum_{i=0}^{2k} \theta_i(y)$ with $y \in \RR$.
Since $\gamma_i > 0$, 
it follows that $h(y)$ is strictly increasing with
\[
\lim_{y \to \infty} h(y) = k \pi + \pi / 2,
\, \, \, \, \, 
\lim_{y \to - \infty} h(y) = - (k + 1) \pi + \pi / 2.
\]
Hence the equation 
\[
h(y) \equiv \frac{\pi}{2}
\, \, \, \, \, 
(\mod \, \, \pi)
\]
possesses $2k$ distinct real roots, as desired.
\, \, \, \, \, \, \, \, \, \, 
\, \, \, \, \, \, \, \, \, \, 
\, \, \, \, \, 
\end{proof}

\begin{Example}
{\em
Let $G$ be a finite connected graph on the vertex set
$V(G) = \{ 1, \ldots, n \}$ with $E(G)$
its edge set. 
We assume that $G$ possesses no loop and no multiple edge.
Let $\eb_1, \ldots, \eb_d$
denote the canonical unit vectors of $\RR^d$.
For an edge $e = \{ i, j \}$ of $G$ with $i < j$,
we define $\rho(e)$ and $\mu(e)$ of $\RR^d$
by setting $\rho(e) = \eb_i - \eb_j$ and
$\mu(e) = \eb_j - \eb_i$.
Write $\Pc_G \subset \RR^d$ for the convex polytope
which is the convex hull of
$\{ \rho(e) : e \in E(G) \} 
\cup \{ \mu(e) : e \in E(G) \}$.
Let $\Hc \subset \RR^d$ denote the hyperplane
defined by the equation $\sum_{i=1}^{d} x_i = 0$.
Then $\Pc_G \subset \Hc$.
Identifying $\Hc$ with $\RR^{d - 1}$, 
it turns out that $\Pc \subset \RR^{d-1}$
is a Fano polytope. 
It then follows from the theory of unimodular matrices
(Schrijver \cite{Sch}) 
that $\Pc_G \subset \RR^{d-1}$ is a Gorenstein Fano polytope.
One of the research problems is to find a combinatorial
characterization of the finite graphs $G$ for which
all root $a \in \CC$ of $i(\Pc_G, n)$ satisfy
$\Re(a) = - 1 / 2$.

For example, 
if $C$ is a cycle of length $6$, then
all roots $a \in \CC$ of $i(\Pc_C, n)$ satisfy
$\Re(a) = - 1 / 2$.
However, 
if $C$ is a cycle of length $7$, then
there is a root $a \in \CC$ of $i(\Pc_C, n)$ with
$\Re(a) \neq - 1 / 2$.

If $G$ is a tree, 
then $\Pc_G$ is unimodular isomorphic to
the regular unit crosspolytope 
which is the convex hull of 
$\{ \pm \eb_1, \ldots, \pm \eb_d \}$
in $\RR^d$.
Hence the $\delta$-vector of $\Pc_G$ is
$\delta(\Pc_G) 
= ({d \choose 0}, {d \choose 1}, \ldots, {d \choose d})$.
Thus by using \cite{RV} again
all roots $a \in \CC$ of $i(\Pc_G, n)$
satisfy $\Re(a) = - 1 / 2$.

Let $G$ be a complete bipartite graph 
of type $(2, d - 2)$.
Thus the edges of $G$ are either $\{ 1, j \}$ 
or $\{ 2, j \}$ with $3 \leq j \leq d$.  
Let $\delta(\Pc) 
= (\delta_0, \delta_1, \ldots, \delta_d)$.
Then
\[
\sum_{k=0}^{d} \delta_k x^k
= (1 + x)^{d - 3} (1 + 2(d - 3)x + x^2).
\]
It has been conjectured that all roots $a \in \CC$ of
$i(\Pc_G, n)$ satisfy $\Re(a) = - 1 / 2$.   
}
\end{Example}

\end{document}